\documentclass[a4paper,notitlepage,twoside,leqno,10pt]{amsart}

\usepackage{bbm,pifont,latexsym}

\usepackage{dcolumn,indentfirst}
\usepackage[bookmarksopen=true,linktocpage=true,pdfstartview={XYZ null null 1.25}]{hyperref}
\usepackage{amsmath,amssymb,amscd,amsthm,amsfonts,mathrsfs}
\usepackage{color,graphicx,xcolor,graphics,subfigure,extarrows,caption2}
\usepackage{titletoc}

\newtheorem{thm}{Theorem}[section]

\newtheorem{cor}[thm]{Corollary}
\newtheorem{prop}[thm]{Proposition}

\newtheorem{thmx}{Theorem}

\theoremstyle{definition}

\newtheorem{exam}{Example}

\theoremstyle{remark}

\newcommand{\EC}{\widehat{\mathbb{C}}}

\newcommand{\C}{\mathbb{C}}

\newcommand{\N}{\mathbb{N}}

\newcommand{\R}{\mathbb{R}}

\newcommand{\ii}{\textup{i}}

\usepackage{enumerate,enumitem}

\begin{document}

\author{FEI YANG}
\address{Department of Mathematics, Nanjing University, Nanjing, 210093, P. R. China}
\email{yangfei@nju.edu.cn}

\title{A criterion to generate carpet Julia sets}

\begin{abstract}
It was known that the Sierpi\'{n}ski carpets can appear as the Julia sets in the families of some rational maps. In this article we present a criterion that guarantees the existence of the carpet Julia sets in some rational maps having exactly one fixed (super-) attracting or parabolic basin. We show that this criterion can be applied to some well known rational maps, such as McMullen maps and Morosawa-Pilgrim family. Moreover, we give also some special examples whose Julia sets are Sierpi\'{n}ski carpets.
\end{abstract}



\date{\today}



\maketitle


\section{Introduction}

Let $\EC$ be the Riemann sphere. According to \cite{Why58}, a set $S\subset\EC$ is called a \emph{Sierpi\'{n}ski carpet} (\emph{carpet} in short) if $S$ is compact, connected, locally connected, has empty interior and has the property that the complementary domains are bounded by pairwise disjoint simple closed curves.
For any rational map $f:\EC\to\EC$ with degree at least two, the \emph{Fatou set} $F(f)$ of $f$ is defined as the maximal open subset on $\EC$ in which the sequence of the iterates $\{f^{\circ n}\}_{n\geq 0}$ forms a normal family in the sense of Montel. The \emph{Julia set} $J(f)$ of $f$ is the complement of the Fatou set in $\EC$. Each connected component of the Fatou set is called a \emph{Fatou component}.

It was known that the Julia sets of all rational maps are compact and have empty interior (if the Julia set is not equal to $\EC$, see \cite[Corollary 4.11]{Mil06}), and in some cases (for example, when the rational maps are post-critically finite), the Julia sets are also connected and locally connected. Hence an interesting question on the topology of the Julia sets of the rational maps is: could a Julia set be a Sierpi\'{n}ski carpet? The answer to this question is not trivial since it is not easy to verify some conditions that a Sierpi\'{n}ski carpet should be satisfied, especially the properties that the boundaries of all Fatou components are simple closed curves and that the intersection of the closure of any two Fatou components is empty. Obviously, the Julia set of any polynomial cannot be a Sierpi\'{n}ski carpet. Indeed, any polynomial $f$ with $\deg(f)\geq 2$ has at infinity a super-attracting fixed point and the Julia set is its immediate basin of attraction (see \cite[Corollary 4.12]{Mil06}). If $f$ has a bounded Fatou component then the closure of this component must intersect with the closure of the basin of infinity. On the other hand, if $f$ does not have any bounded Fatou components, then the boundary of the basin of infinity cannot be a simple closed curve and hence its Julia set is none a Sierpi\'{n}ski carpet.

For the carpet Julia sets of rational maps, the first example was found by Milnor and Tan Lei \cite[Appendix F]{Mil93}. Their example was a quadratic rational map with two periodic super-attracting cycles and one of the period is $3$ and the other one is $4$. Later, in his Ph.d thesis Pilgrim found a carpet Julia set of a cubic rational map
\cite[\S 5.6]{Pil94} by considering the branched covering of the sphere to itself. As the Julia sets of rational maps, the Sierpi\'{n}ski carpets appeared in many literatures from then on. For example, McMullen maps (see \cite{DLU05}), Morosawa-Pilgrim family (see \cite{Ste08}) and the generalized McMullen maps (see \cite{XQY14}).

One may wonder why the Sierpi\'{n}ski carpets attracted many people's interest. On the one hand, the appearance of the carpet Julia sets reveals the richness of the dynamics of rational maps. In fact, how to classify the carpet Julia sets dynamically is still an open problem (one may refer to \cite{DP09} for a partial answer to this question). On the other hand, in the geometry group theory, the quasisymmetric equivalences of the Sierpi\'{n}ski carpets has been studied extensively since it was partially motivated by the Kapovich-Kleiner conjecture \cite{KK00}. Recently, the quasisymmetric geometry of the carpet Julia sets has also been considered in \cite{BLM16}, \cite{QYZ14} and \cite{QYY16}.

In view of the fact that the previous proof methods for the existence of the carpet Julia sets are different (indeed, each method works only for one special family of rational maps),
a natural question is: does there exist a general method to judge whether a given rational map has a carpet Julia set? As a partial answer to this question, we give a useful criterion in this article and hence shed some lights on this problem.

\begin{thmx}\label{thm:main}
Let $f$ be a rational map satisfying the following conditions:
\begin{enumerate}
\item The Julia set of $f$ is connected;
\item There exist two different Fatou components $U$ and $V$ of $f$ such that $f(U)=U$ and $f^{-1}(U)=U\cup V$;
\item If $W$ is a connected component of $f^{-1}(V)$, then $\deg(f|_W)>\deg(f|_U)$; and
\item The intersection of any critical orbit with $U$ is non-empty; In particular, $f$ has degree at least $3$.
\end{enumerate}
Then the Julia set of $f$ is a Sierpi\'{n}ski carpet.
\end{thmx}

From the assumptions (b) and (d) in Theorem \ref{thm:main}, all the critical points of $f$ are iterated eventually to the unique fixed Fatou component $U$, which is either attracting, super-attracting or parabolic. The key of the proof of Theorem \ref{thm:main} is to show that the immediate basin of the unique fixed Fatou component of $f$ is a Jordan domain.

We will give some examples to show that each of the assumptions in Theorem \ref{thm:main} is necessary. Specifically, we prove that there exist rational maps whose Julia sets are not Sierpi\'{n}ski carpets provided only three of four conditions in Theorem \ref{thm:main} are satisfied. This means that our criterion in Theorem \ref{thm:main} has some kind of sharpness.

It seems that the condition (c) is a bit technical. Actually an intuitional idea can be explained as following: In order to prove that $U$ is bounded by a simple closed curve, we need to show that $\EC\setminus\overline{U}$ consists of exactly one connected component first. Assume in the contrary that $\EC\setminus\overline{U}$ contains two \textit{different} components $A_V$ and $A_W$ which contain $V$ and $W$ respectively, where $f(W)=V$. Since the local degree cannot exceed the global degree, it means that $\deg(f|_W)\leq \deg(f|_{A_W})=\deg(f|_{\partial A_W})\leq \deg(f|_U)$. However this violates the condition (c) (see the proof of Theorem \ref{thm:main} in \S\ref{sec-proof} for details).

\vskip0.2cm
Let us give some applications of Theorem \ref{thm:main}. According to the assumptions in Theorem \ref{thm:main}, we assume that $\deg(f|_U:U\to U)=d_\infty$ and $\deg(f)=d_0+d_\infty$, where $d_0\geq 1$ and $d_\infty\geq 2$. Since $f^{-1}(U)=U\cup V$, we have $\deg(f|_V:V\to U)=d_0$. By a standard quasiconformal surgery (based on the connectivity of the Julia sets), one can assume that each of the Fatou components $U$ and $V$ contains at most one critical point (counted without multiplicity). This means that $U$ is a fixed super-attracting basin or a parabolic basin. To simplify the expressions of the rational maps, without loss of generality, we assume that $f$ has a unique fixed super-attracting basin\footnote{Theorem \ref{thm:main} can be applied to the parabolic rational maps of course. But the main point here is to give an application of a canonical family of hyperbolic rational maps. See \S\ref{subsec-para} for an example of parabolic rational map with a carpet Julia set.} centered at $\infty$, $f^{-1}(\infty)=\{\infty, 0\}$ and $f^{-1}(0)=\{b_1,\cdots,b_n\}$, where $n\geq 1$ and $b_1=1$. This means that $f$ has the form
\begin{equation}\label{eq:family}
f_\lambda(z)=\frac{\lambda}{z^{d_0}}\prod_{i=1}^n (z-b_i)^{d_i},
\end{equation}
where $\sum_{i=1}^n d_i=d_0+d_\infty$, $d_0\geq 1$ and $\lambda,b_i\in\C\setminus\{0\}$. Based on the condition (c) in Theorem \ref{thm:main}, we require further that $d_i\geq d_\infty+1\geq 3$ for all $1\leq i\leq n$. As an immediate corollary of Theorem \ref{thm:main}, we have

\begin{cor}\label{cor:carpet-1}
The Julia set of $f_\lambda$ is a Sierpi\'{n}ski carpet provided it is connected and all critical points of $f_\lambda$ escape to the infinity under iteration.
\end{cor}

We would like to remind the reader that the different $b_i$'s are not necessarily in different Fatou components. Indeed, if $b_i$ and $b_j$ are contained in the same Fatou component $W_i$, then $\deg(f|_{W_i}:W_i\to V)\geq d_i+d_j\geq d_\infty+1=\deg(f|_U:U\to U)+1$, which still satisfies the condition (c).

\vskip0.2cm
Let us consider the simplest case first. Set $n=1$ in \eqref{eq:family}. This means that $d_1=d_0+d_\infty\geq 3$ and the uni-parametric family $f_\lambda$ writes as
\begin{equation}\label{equ:F-lambda}
F_\lambda(z)=\lambda\,\frac{(z-1)^{d_0+d_\infty}}{z^{d_0}}, \text{ where }\lambda\in\C\setminus\{0\}.
\end{equation}
Note that $F_\lambda$ has the critical orbit $1\overset{(d_0+d_\infty)}{\longmapsto} 0\overset{(d_0)}{\longmapsto}\infty\overset{(d_\infty)}{\longmapsto}\infty$ and a free critical point $z_\lambda=-\frac{d_0}{d_\infty}$ with local degree $2$. As an application of Corollary \ref{cor:carpet-1}, we have

\begin{prop}\label{prop:carpet-3}
The Julia set of $F_\lambda$ is a Sierpi\'{n}ski carpet if $z_\lambda\not\in V$ but $F_\lambda^{\circ k}(z_\lambda)$ $\in V$ for some $k\geq 2$, where $V$ is the Fatou component of $F_\lambda$ containing $0$.
\end{prop}

According to \cite{DLU05} and \cite{Ste08}, McMullen maps and Morosawa-Pilgrim family, respectively, are defined by
\begin{equation}\label{eq:Mc-Mor-Pil}
g_\mu(z)=z^{d_\infty}+\frac{\mu}{z^{d_0}} \text{\quad and \quad} h_\nu(z)=\nu\left(1+\frac{4z^3}{27(1-z)}\right),
\end{equation}
where $d_0\geq 1$, $d_\infty\geq 2$ and $\mu, \nu\in\C\setminus\{0\}$ are parameters. A straightforward calculation shows that the map $z\mapsto -\frac{1}{\mu}z^{d_0+d_\infty}$ semi-conjugates $g_\mu$ to $F_\lambda$ with $\lambda=(-\mu)^{d_\infty-1}$. Therefore, the free critical point $z_\lambda=-\frac{d_0}{d_\infty}$ of $F_\lambda$ corresponds to $d_0+d_\infty$ free critical points
$c_j=(\tfrac{d_0}{d_\infty}\mu)^{1/(d_0+d_\infty)}e^{2\pi\ii j/(d_0+d_\infty)}$ of $g_\mu$, where $1\leq j\leq d_0+d_\infty$. On the other hand, $h_\nu$ has the critical orbits
\begin{equation*}
\tfrac{3}{2}\overset{(2)}{\longmapsto}0\overset{(3)}{\longmapsto}\nu\longmapsto\cdots \text{ and } \infty\overset{(2)}{\longmapsto}\infty.
\end{equation*}
As an immediate corollary of Proposition \ref{prop:carpet-3}, we have the following result.

\begin{cor}[{Devaney-Look-Uminsky, Steinmetz}]\label{cor:McM-Mor-Pil}
Let $U_\mu$ (resp. $U_\nu$) be the immediate super-attracting basin of $\infty$ of $g_\mu$ (resp. $h_\nu$).
\begin{enumerate}
\item If there exists an integer $k\geq 2$ such that $g_\mu^{\circ k}(c_j)\in g_\mu^{-1}(U_\mu)\setminus U_\mu$ for some $1\leq j\leq d_0+d_\infty$, then the Julia set of $g_\mu$ is a Sierpi\'{n}ski carpet.
\item If $\nu\in h_\nu^{-1}(U_\nu)\setminus U_\nu$, then the Julia set of $h_\nu$ is a Sierpi\'{n}ski carpet.
\end{enumerate}
\end{cor}

Let us move to a bit complicated example. Set $n=2$, $d_0=4$, $d_\infty=2$ and $d_1=d_2=3$ in \eqref{eq:family}. We consider the following one-dimensional family of rational maps
\begin{equation*}
G_c(z)=\frac{a(z-1)^3(z-b)^3}{z^4}, \text{ where }
a=\frac{c(c-4)^3}{27(c-1)^6},~
b=\frac{c(1+2c)}{4-c}
\end{equation*}
and $c\in\C\setminus\{0,1,4,-\tfrac{1}{2}\}$ is the parameter. A direct calculation shows that $G_c$ has the critical orbits $b\overset{(3)}{\longmapsto} 0\overset{(4)}{\longmapsto}\infty\overset{(2)}{\longmapsto}\infty$, $c\overset{(2)}{\longmapsto} 1\overset{(3)}{\longmapsto} 0$ and a free critical point $z_c=\frac{2(1+2c)}{c-4}$ with local degree $2$. As another application of Corollary \ref{cor:carpet-1}, we have

\begin{prop}\label{prop:carpet-2}
The Julia set of $G_c$ is a Sierpi\'{n}ski carpet if $z_c\not\in U$ but $G_c^{\circ k}(z_c)\in U$ for some $k\geq 1$, where $U$ is the immediate super-attracting basin of $G_c$ at $\infty$.
\end{prop}

In particular, if $G_c^{\circ k}(z_c)=0$ for some $k\geq 1$, then $G_c$ is post-critically finite and the Julia set of $G_c$ is connected and hence a Sierpi\'{n}ski carpet by Corollary \ref{cor:carpet-1}. See Figure \ref{Fig_Julia-carpet} for the parameter plane of $G_c$ and a carefully chosen Julia set.

\begin{figure}[!htpb]
  \setlength{\unitlength}{1mm}
  \includegraphics[height=58mm]{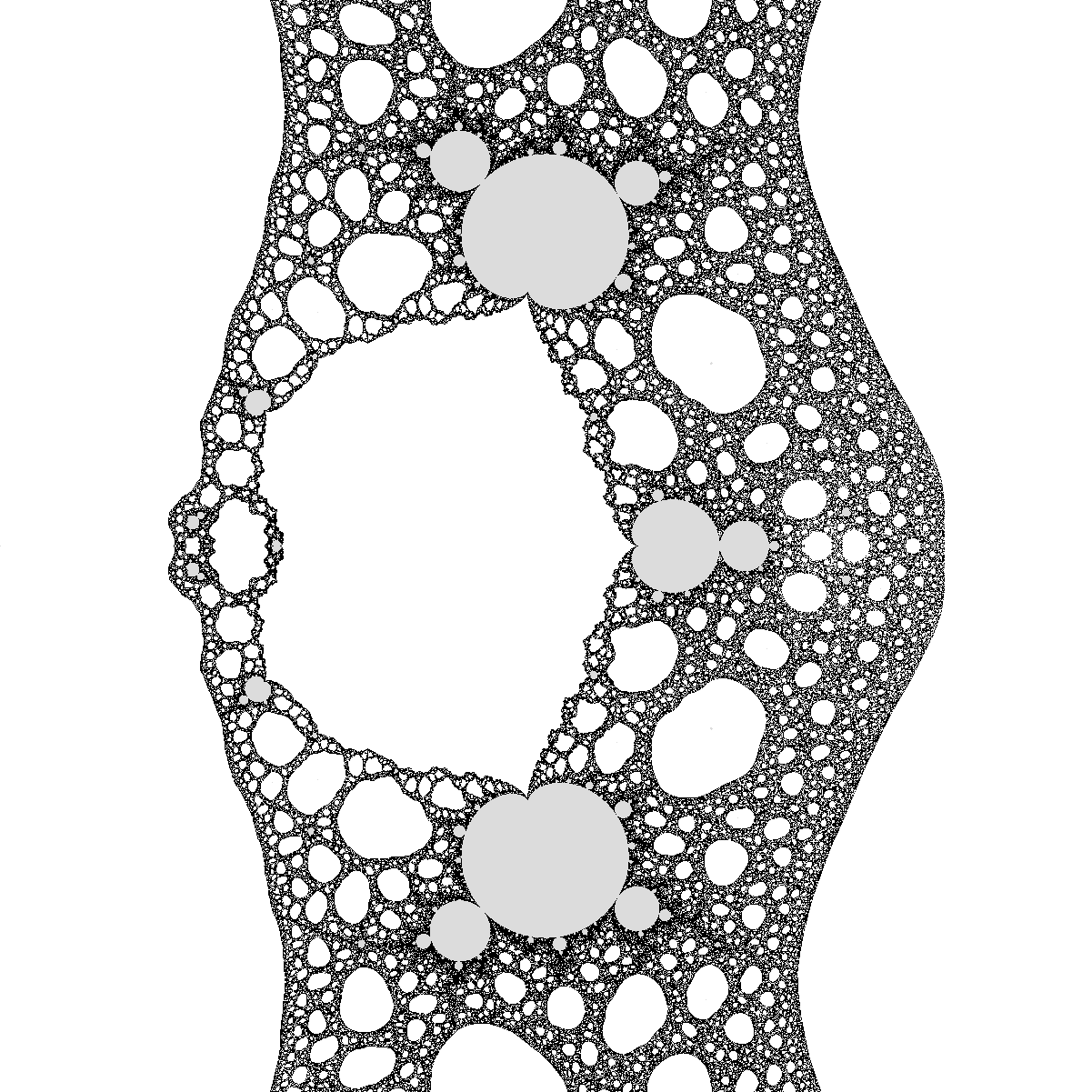}\quad
  \includegraphics[height=58mm]{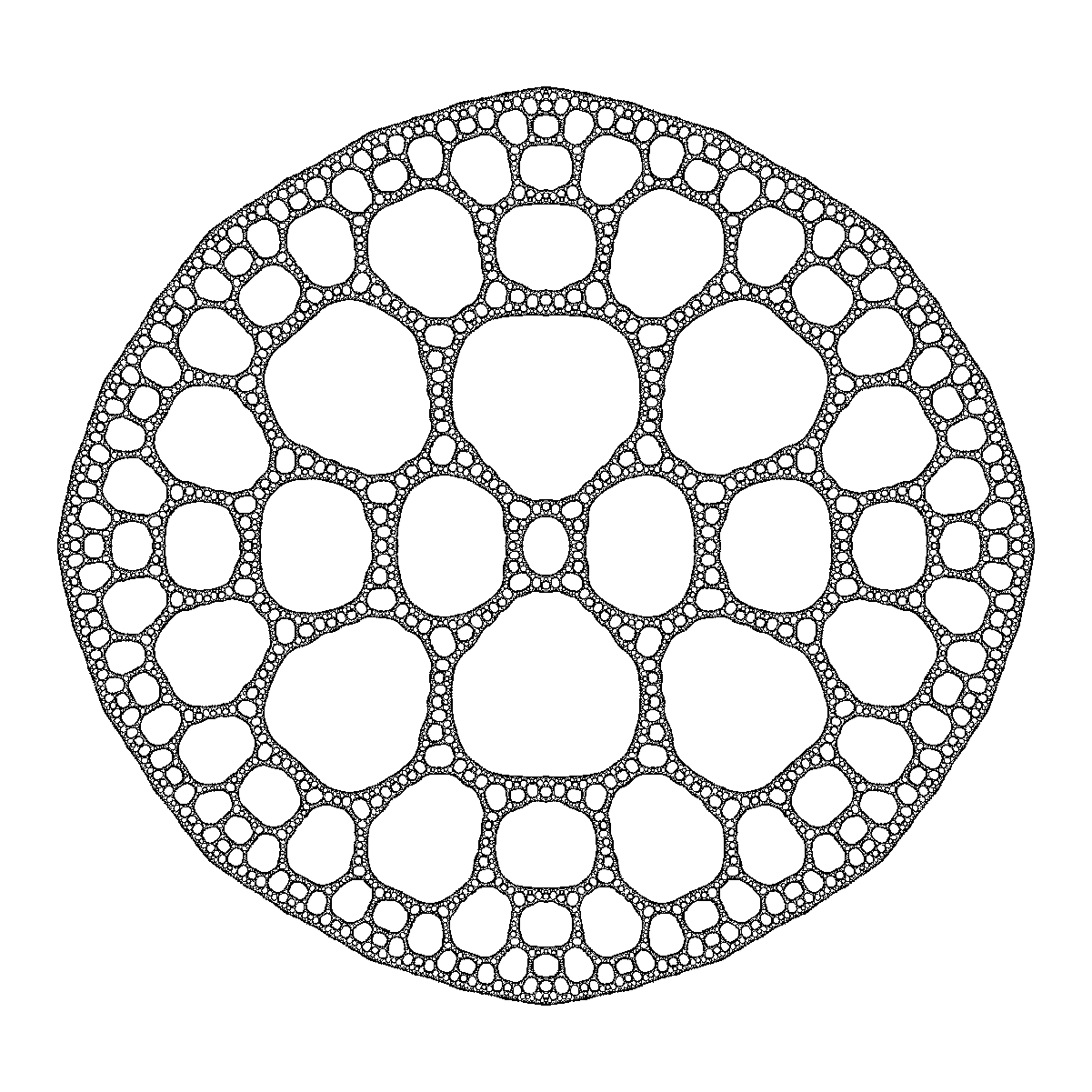}
  \caption{Left: The parameter plane of $G_c$, where almost all of the white ``holes" correspond to the parameters such that the corresponding Julia sets are Sierpi\'{n}ski carpets (The exceptional cases are the hyperbolic components containing $1$ and $-1/2$ where the parameters correspond to the disconnected Julia sets). Right: The Julia set of $G_{c_0}$, where $c_0=\ii\sqrt{2}$ is chosen such that $G_{c_0}(z_{c_0})=1$ and the Julia set of $G_{c_0}$ is a Sierpi\'{n}ski carpet.}
  \label{Fig_Julia-carpet}
\end{figure}

Note that our criterion in Theorem \ref{thm:main} works only for the rational maps with exactly one fixed Fatou component. Actually, there are many carpet Julia sets whose corresponding Fatou sets have periodic Fatou components with period greater than one. We refer the reader to \cite{DFGJ14} for a comprehensive study on the existence of carpet Julia sets in the quadratic rational maps. However, as far as we know, all the carpet Julia sets constructed before correspond to the rational maps whose periodic Fatou components have period one or period at least \textit{three}. Indeed, in \cite{DFGJ14} the authors deal with the quadratic rational maps but the period is always bigger than or equal to 3. On the other hand, it has been proved in \cite{AY09} that there are no carpet Julia sets for quadratic rational maps with period 2 Fatou components (and of course there are none with period 1 since they are conjugated to polynomials). Hence a natural question is: Does there exist a carpet Julia set of higher degree rational map whose corresponding Fatou set contains a Fatou component with period $2$? In this article we give an affirmative answer to this question.

\begin{thmx}\label{thm:peri-2}
There exists a family of cubic rational maps whose Julia sets are Sierpi\'{n}ski carpets and whose Fatou sets contain Fatou components of period $2$.
\end{thmx}

In fact, the Julia set of a rational map could be a Sierpi\'{n}ski carpet if it contains some critical points. In particular, there exists infinitely renormalizable rational map whose Julia set is a Sierpi\'{n}ski carpet. For more details, see \cite{QYZ14} and \cite{QYY16}.

This article is organized as following: In \S\ref{sec-proof}, we give the proofs of Theorem \ref{thm:main}, Propositions \ref{prop:carpet-3} and \ref{prop:carpet-2}. Then some specific examples will be analysed and used to show the necessity of every assumption in Theorem \ref{thm:main}. In \S\ref{sec-special}, we will give some examples of special carpet Julia sets, such as a carpet Julia set with a parabolic fixed point and a carpet Julia set with a periodic Fatou component with period $2$ and hence prove Theorem \ref{thm:peri-2}.

\section{Proofs of Theorems and some examples}\label{sec-proof}

In this section, we first give the proofs of Theorem \ref{thm:main} and its applications Propositions \ref{prop:carpet-3} and \ref{prop:carpet-2}. Then we use some examples to show that all the conditions in Theorem \ref{thm:main} are necessary.

\subsection{Proofs of Theorems}

As stated in the introduction, in order to prove that a Julia set is a Sierpi\'{n}ski carpet, the main point is to show that all the Fatou components are bounded by pairwise disjoint simple closed curves (i.e. Jordan curves). For Theorem \ref{thm:main}, our strategy is to show first that the fixed Fatou component $U$ is bounded by a Jordan curve. By taking preimages, one can obtain that all the Fatou components are bounded by Jordan curves. Next we prove that $U$ and its preimage $V$ are bounded by \textit{disjoint} curves. The last step is to show that all the Fatou components are bounded by pairwise disjoint curves.

\begin{proof}[{Proof of Theorem \ref{thm:main}}]
Suppose that $f$ satisfies all the assumptions in Theorem \ref{thm:main}. By (b) and (d), we know that all the critical points are contained in the Fatou set of $f$, the Fatou component $U$ is attracting, super-attracing or parabolic and this is the unique periodic Fatou component of $f$. From \cite{TY96} and by (a), the Julia set of $f$ is connected and locally connected. By (c), we have the forward orbit of the Fatou components $W\to V\to U\to U$ under $f$. Without loss of generality, we assume that $\infty\in U$.

According to Torhorst's theorem \cite[p.\,106]{Why42}, the boundary of each Fatou component of $f$ is locally connected. In particular, $U$ is simply connected and $\partial U$ is locally connected. We claim that $U$ is a Jordan domain, i.e. $U$ is bounded by a simple closed curve. According to \cite[Proposition 2.5]{Pil96}, each component of $\EC\setminus\overline{U}$ is bounded by a Jordan curve. Let $A_V$ and $A_W$ be the connected components of $\EC\setminus\overline{U}$ containing $V$ and $W$ respectively, where $W$ is a connected component of $f^{-1}(V)$. We claim that $A_V=A_W$. Assuming the contrary that $A_V\neq A_W$, we have $f^{-1}(U)=U\cup V\subset U\cup A_V$ by (b) and hence $f(A_W)\subset\EC\setminus\overline{U}$. Since $A_W$ is a Jordan domain, $f(A_W)$ is connected and $f(\partial A_W)\subset\partial U$, it follows that $f(A_W)$ is exactly a component of $\EC\setminus\overline{U}$. This means that $f(A_W)=A_V$ since $f(W)=V$. Then $\deg(f|_W:W\to V)\leq \deg(f|_{\partial A_W}:\partial A_W\to \partial A_V)\leq \deg(f|_{\partial U}:\partial U\to \partial U)= \deg(f|_U:U\to U)$. However, by the condition (c) we have $\deg(f|_W:W\to V)>\deg(f|_U:U\to U)$. This is a contradiction and hence we have $A_V=A_W$ and both $V$ and $W$ are contained in $A_V$. In particular, $f^{-1}(V)$ is contained in $A_V$. Now it is easy to see that $\EC\setminus\overline{U}$ consists of exactly one connected component $A_V$. Indeed, if $\EC\setminus\overline{U}$ contains a component $A'\neq A_V$, then by (d) there exists an integer $k\geq 1$ such that $f^{\circ k}(A')=A_V$. This means that each point in $V$ would have more than $\deg(f)$ preimages, which is a contradiction. Although $\EC\setminus\overline{U}=A_V$ is a Jordan domain, it is still not enough to conclude that $\partial U$ is a Jordan curve. We need to exclude the existence of the ``hairs" attaching on the boundary of $A_V$.

Since $\partial U$ is locally connected and $\partial A_V\subset \partial U$, it means that each $z\in\partial A_V$ is accessible from $U$. There exist two simple (open) arcs $\gamma_1,\gamma_2\subset U$ landing at $z$ such that $\gamma_1$ and $\gamma_2$ connect $z$ with $\infty$ and $\gamma_1\cap \gamma_2=\emptyset$. Then $\Gamma:=\gamma_1\cup\gamma_2\cup\{z,\infty\}$ divides $\EC$ into two Jordan domains. Let $A$ be the connected component of $\EC\setminus\Gamma$ such that $A\cap \overline{A}_V=\emptyset$. Then $\{f^{\circ n}:n\in\N\}$ is a normal family in $A$ since\footnote{Note that $A\cap F(f)$ is contained in $U$, where $F(f)$ is the Fatou set of $f$.} each $f^{\circ n}(A)$ does not take the values in $V$. This means that $A$ is contained in the Fatou set. By the arbitrariness of $z$ and $\Gamma$, it means that $\partial U=\partial A_V$ and $U$ is a Jordan domain.

Since all the critical points are contained in the Fatou set of $f$, according to \cite[Proposition 2.8]{Pil96}, each preimage of $U$ is a Jordan domain. By induction, we know that all the Fatou components of $f$ are Jordan domains.

It remains to prove that each pair of the boundaries of the Fatou components are disjoint. To see this, suppose that $z\in\partial U\cap \partial V$. Then there exists a simple arc $\beta\subset U$ without touching the critical values of $f$ (they are finite in number) and landing at $f(z)$. Since $f^{-1}(U)=U\cup V$, it follows that $f^{-1}(\beta)$ has two connected components $\beta_1\subset U$ and $\beta_2\subset V$ and both of them land at $z$. This means that $z$ is a critical point of $f$. From the assumption (d), this is a contradiction. Hence $\partial U\cap \partial V=\emptyset$.

Let $\alpha_1$ and $\alpha_2$, respectively, be the connected components of $f^{-k}(\partial U)$ and $f^{-\ell}(\partial U)$ satisfying $\alpha_1\neq\alpha_2$, where $k,\ell\in\N$. We claim that $\alpha_1\cap\alpha_2=\emptyset$. Otherwise, by iterating $\alpha_1$ and $\alpha_2$ forward in several times, we conclude that $\partial U\cap\partial V\neq\emptyset$ if $k\neq \ell$, and $\partial U$ contains a point in the critical orbit of $f$ if $k=\ell$. However, neither these two cases are possible. Therefore, the boundaries of the Fatou components are pairwise disjoint. This completes the proof that the Julia set of $f$ is a Sierpi\'{n}ski carpet.
\end{proof}

Note that Propositions \ref{prop:carpet-3} and \ref{prop:carpet-2} are applications of Corollary \ref{cor:carpet-1}. From the assumptions in Propositions \ref{prop:carpet-3} and \ref{prop:carpet-2}, it is sufficient to prove that the Julia sets of $F_\lambda$ and $G_c$ are connected.

\begin{proof}[{Proof of Proposition \ref{prop:carpet-3}}]
Note that $F_\lambda$ has a super-attracting basin $U$ centered at the infinity and it has the critical orbit $1\mapsto 0\mapsto\infty\mapsto\infty$ and a free critical point $z_\lambda=-\frac{d_0}{d_\infty}$ with local degree $2$. Let $V$ be the Fatou component of $F_\lambda$ containing $0$. Suppose that $F_\lambda^{\circ k}(z_\lambda)\in V$ for some $k\geq 2$ and $z_\lambda\not\in V$. We claim first that $U\neq V$. If $U=V$, then $F_\lambda^{-1}(U)=U$ and $U$ is completely invariant. From the assumption $F_\lambda^{\circ k}(z_\lambda)\in V=U$ for some $k\geq 2$, we have $z_\lambda\in V$, which contradicts the assumption that $z_\lambda\not\in V$. This proves that claim that $U\neq V$.

Since $z_\lambda\not\in U$, it means that $U$ is a super-attracting basin containing exactly one critical point $\infty$. We claim that $U$ is simply connected. Indeed, let $U_0$ be a small simply connected neighborhood of $\infty$ such that $F_\lambda(\overline{U}_0)\subset U_0$. We use $U_1$ to denote the component of $F_\lambda^{-1}(U_0)$ containing $U_0$. Inductively, let $U_n$ be the component of $F_\lambda^{-n}(U_0)$ containing $U_0$ for all $n\geq 1$. Since $F_\lambda:U_{n+1}\setminus\{0\}\to U_n\setminus\{0\}$ is a covering map for all $n\in\N$, it follows that each $U_n$ is simply connected. Note that $U_0\subset U_1\subset U_2\subset\cdots$ and $U=\bigcup_{n\in\N} U_n$. This means that $U$ is simply connected.

We now prove that all the $m$-th preimages of $U$ are simply connected, where $m\in\N$. Since $z_\lambda\not\in V$, the map $F_\lambda:V\setminus\{0\}\to U\setminus\{\infty\}$ is a covering map. Therefore, $V\setminus\{0\}$ is doubly connected and $V$ is simply connected. Let $W$ be the Fatou component of $F_\lambda$ containing $1$. From the assumption that $F_\lambda^{\circ k}(z_\lambda)\in V$ for some $k\geq 2$, we know that $z_\lambda\not\in W$ and $F_\lambda:W\setminus\{1\}\to V\setminus\{0\}$ is also a covering map. Therefore, $W$ is simply connected also. In summary, we have $z_\lambda\not\in U\cup V\cup W$ and $U$, $V$, $W$ are all simply connected. Applying Riemann-Hurwitz's formula inductively, it is easy to see all the $m$-th preimages of $W$ are simply connected, where $m\in\N$. Hence all the Fatou components of $F_\lambda$ are simply connected and $J(F_\lambda)$ is connected. By Corollary \ref{cor:carpet-1}, $J(F_\lambda)$ is a Sierpi\'{n}ski carpet.
\end{proof}

The proof of Proposition \ref{prop:carpet-2} is similar to that of Proposition \ref{prop:carpet-3} although the argument is slightly complicated. For the completeness, we include a proof here.

\begin{proof}[{Proof of Proposition \ref{prop:carpet-2}}]
Let $U$ be the super-attracting basin of $G_c$ centered at the infinity. Recall that $G_c$ has the critical orbits $b\mapsto 0\mapsto\infty\mapsto\infty$ and $c\mapsto 1\mapsto 0$. Suppose that $G_c^{\circ k}(z_c)\in U$ for some $k\geq 1$ and $z_c\not\in U$, where $z_c=\frac{2(1+2c)}{c-4}$ is the unique free critical point of $G_c$. Let $V$ be the Fatou component of $G_c$ containing $0$. Similar to the proof of Proposition \ref{prop:carpet-3}, we have $U\neq V$ and $U$ is simply connected.

We now prove that all the $m$-th preimages of $U$ are simply connected, where $m\in\N$. Applying Riemann-Hurwitz's formula to $G_c:V\setminus\{0\}\to U\setminus\{\infty\}$, we know that $V\setminus\{0\}$ is doubly connected and $z_c\not\in V$. Therefore, $V$ is simply connected. Let $W_1$ and $W_b$ be the Fatou components containing $1$ and $b$ respectively. There are following two cases:

Case I: $W_1=W_b$. Then we have $G_c:W_1\setminus\{1,b\}=W_1\setminus G_c^{-1}(0)\to V\setminus\{0\}$. If $b\neq 1$ (i.e. $c\neq -2$), then the Euler's characteristic $\chi(W_1\setminus\{1,b\})\leq -1$ since $W_1\neq \EC$. By Riemann-Hurwitz's formula, $W_1$ is simply connected and $z_c\in W_1$. If $b=1$ (i.e. $c=-2$), then $z_c=1$ and we have $G_c(W_1\setminus\{1\})=V\setminus\{0\}$ and $\chi(W_1\setminus\{1\})\leq 0$. By Riemann-Hurwitz's formula, $W_1$ is simply connected also. Inductively, in both cases, all the $m$-th preimages of $W_1$ are simply connected, where $m\in\N$. Hence all the Fatou components are simply connected and $J(G_c)$ is connected.

Case II: $W_1\neq W_b$. Then we have two different maps $G_c:W_1\setminus\{1\}\to V\setminus\{0\}$ and $G_c:W_b\setminus\{b\}\to V\setminus\{0\}$. This means that $W_1$ and $W_b$ are simply connected, $z_c\not\in W_1\cup W_b$ and the restrictions of $G_c$ on $W_1\setminus\{1\}$ and $W_b\setminus\{b\}$ are both covering maps. Since $G_c^{\circ k}(z_c)\in U$, we have $G_c^{\circ (k-2)}(z_c)\in W_1\cup W_b$ and $k\geq 3$. Let $Z_c$ be the Fatou component containing $c$. If $z_c\not\in Z_c$, then it is easy to see that all the $m$-th preimages of $W_1$ and $W_b$ are simply connected, where $m\in\N$. Hence $J(G_c)$ is connected. If $z_c\in Z_c$, then $G_c(z_c)\in W_1$. We have $G_c:Z_c\setminus G_c^{-1}(1)\to W_1\setminus\{1\}$. Note that $z_c\in Z_c\neq \EC$ and $c\in G_c^{-1}(1)$. By Riemann-Hurwitz's formula, it means that the Euler's characteristic $\chi(Z_c\setminus G_c^{-1}(1))= -1$ and $z_c\in Z_c\setminus G_c^{-1}(1)$ or $\chi(Z_c\setminus G_c^{-1}(1))= 0$ and $z_c=c$. In both cases, $Z_c$ is simply connected. Inductively, all the $m$-th preimages of $W_b$ and $W_1$ are simply connected, where $m\in\N$. Hence all the Fatou components are simply connected and $J(G_c)$ is connected.
\end{proof}

\subsection{The counter-examples}

In this subsection, in order to verify that each of the assumptions in Theorem \ref{thm:main} is necessary, we give some examples to show that the corresponding Julia sets could not be Sierpi\'{n}ski carpets if only three of four conditions in Theorem \ref{thm:main} are satisfied.

\begin{exam}\label{exam-1}
We show that there exists a family of rational maps $f_1$ satisfying the conditions (b) (c) (d) in Theorem \ref{thm:main} but $J(f_1)$ is not a Sierpi\'{n}ski carpet. Consider McMullen maps \begin{equation*}
f_1(z)=z^{d_\infty}+\lambda_1/z^{d_0},
\end{equation*}
where $1/d_0+1/d_\infty<1$ and $\lambda_1\in\C\setminus\{0\}$. According to \cite{DLU05}, if $\lambda_1\neq 0$ is small enough, then there exist three different Fatou components $U$, $V$ and $W$ such that $f_1(U)=U$, $f_1^{-1}(U)=U\cup V$ and $f_1^{-1}(V)=W$, where $U$ is the super-attracting basin centered at $\infty$ and $V\neq U$ is the Fatou component containing the origin. However, the Julia set of $f_1$ is a Cantor set of circles, which is not connected and hence not a Sierpi\'{n}ski carpet (see Figure \ref{Fig_para-carpet-1}, where $d_0=d_\infty=3$ and $\lambda_1=10^{-2}$).
\end{exam}

\begin{exam}\label{exam-2}
We now construct a rational map $f_2$ satisfying the conditions (a) (c) (d) in Theorem \ref{thm:main} but $J(f_2)$ is not a Sierpi\'{n}ski carpet. Let
\begin{equation*}
f_2(z)=-\frac{3(3z-4)}{2z(2z-3)^2}.
\end{equation*}
The critical orbit of $f_2$ is $1\overset{(3)}{\longmapsto} \tfrac{3}{2}\overset{(2)}{\longmapsto}\infty\overset{(2)}{\longmapsto}0\overset{(1)}{\longmapsto}\infty$. Since $f_2$ is post-critically finite, the Julia set of $f_2$ is connected. Let $U$, $V$ and $W$ be the Fatou components containing $\infty$, $\tfrac{3}{2}$ and $1$ respectively. It is easy to see that $f_2$ satisfies all the conditions in Theorem \ref{thm:main} except the second. By a straightforward calculation, one can verify that the restriction of $f_2$ on $(-\infty,0)$ is strictly decreasing and it is a homeomorphism to itself. Therefore, there exists a unique point $\xi\in(-\infty,0)$ such that $f_2(\xi)=\xi$ which is repelling. This means that $(-\infty,\xi)\subset U$ and $(\xi,0)\subset Z$, where $Z$ is the Fatou component containing $0$. Hence $J(f_2)$ is not a Sierpi\'{n}ski carpet since $\xi\in \overline{U}\cap \overline{Z}$ (see Figure \ref{Fig_para-carpet-1}).
\end{exam}

\begin{figure}[!htpb]
  \setlength{\unitlength}{1mm}
  \includegraphics[height=55mm]{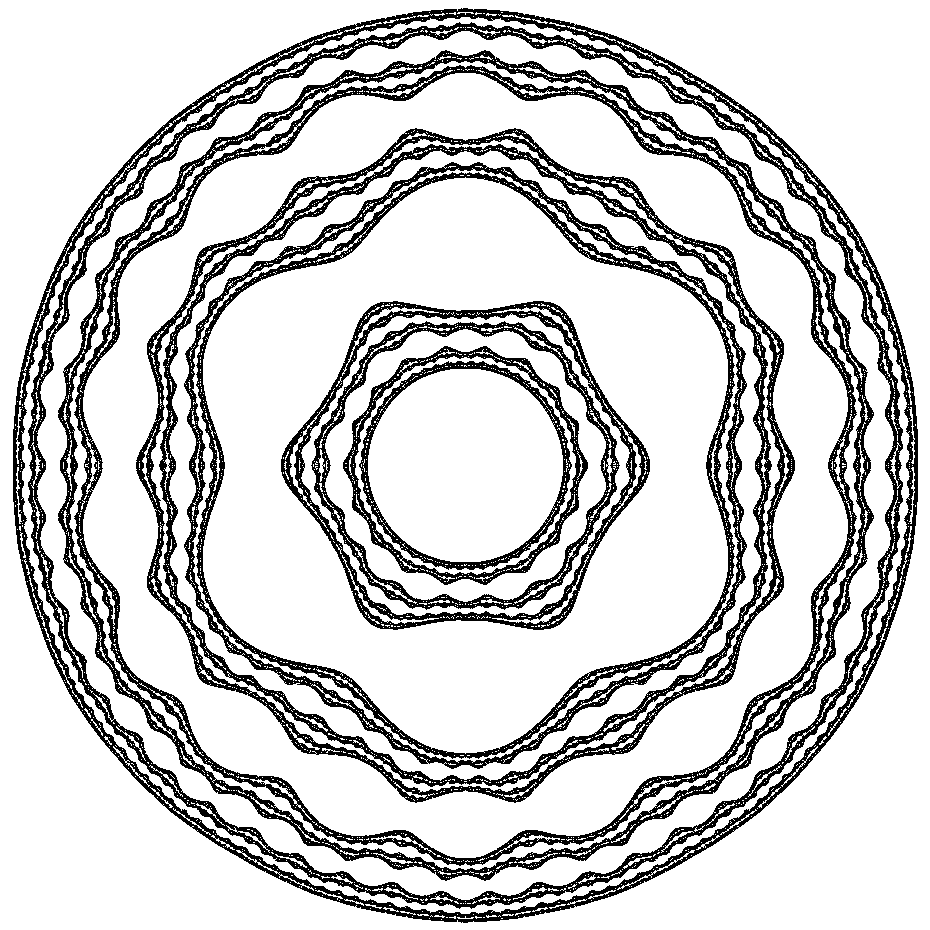}\quad
  \includegraphics[height=55mm]{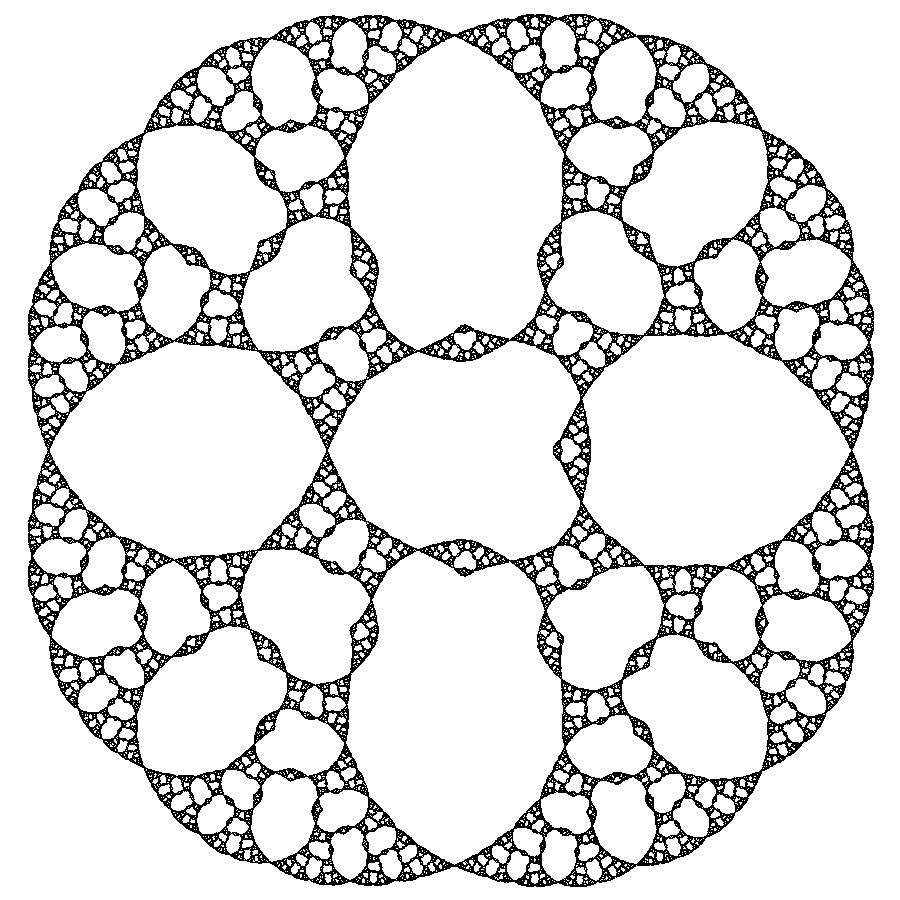}\vskip0.2cm
  \includegraphics[height=55mm]{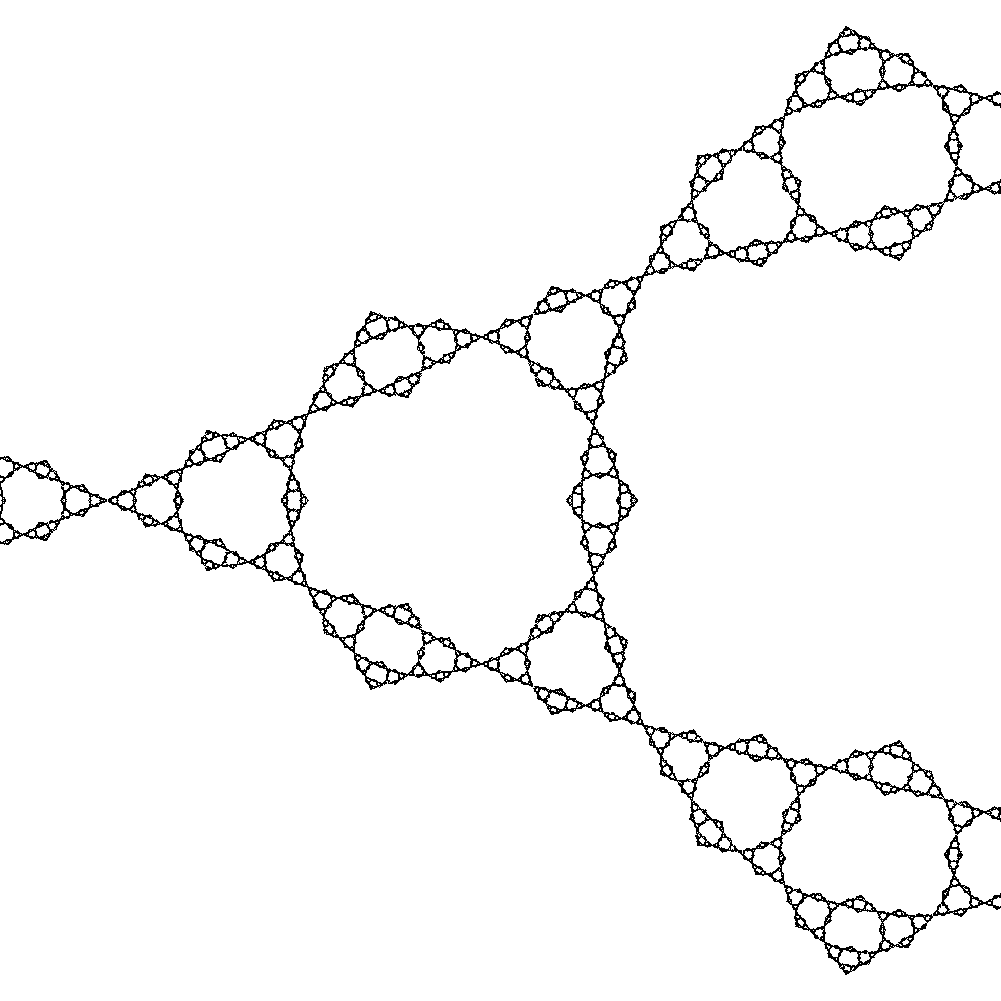}\quad
  \includegraphics[height=55mm]{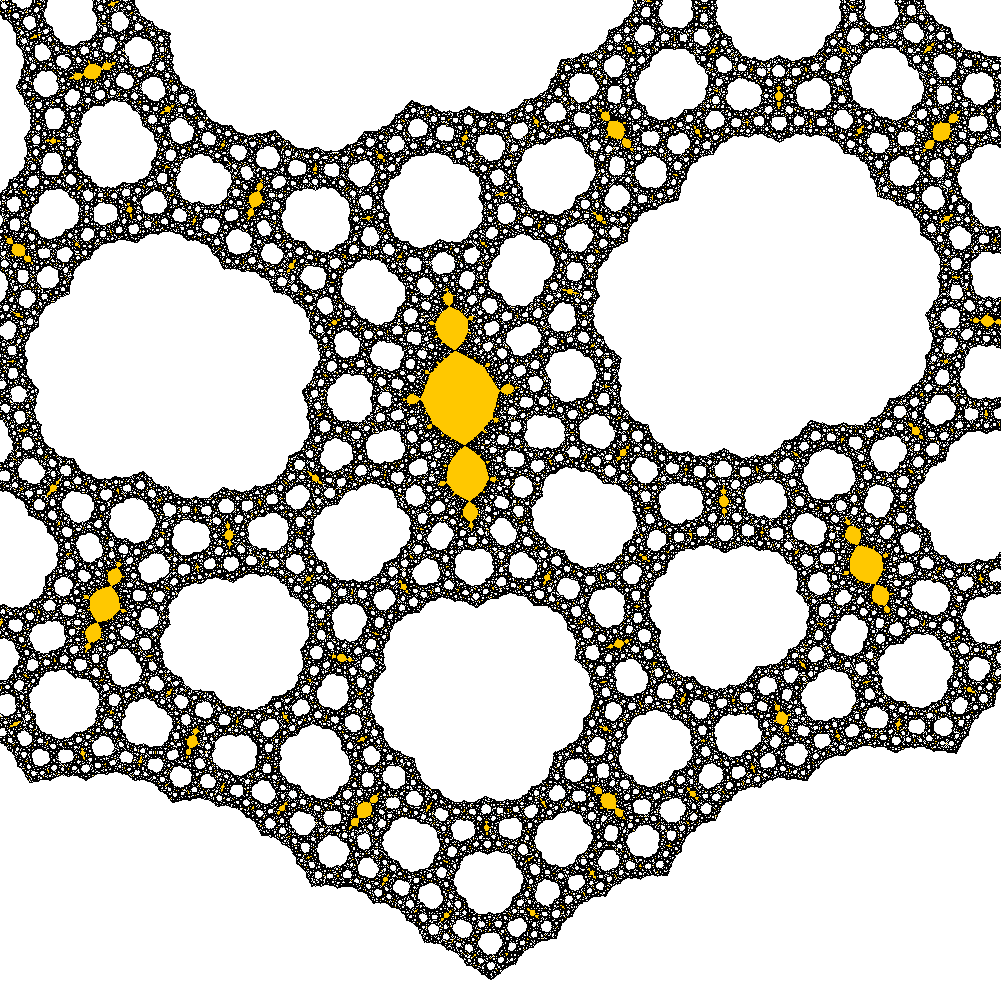}
  \caption{The Julia sets of $f_1$, $f_2$, $f_3$ and $f_4$. The first Julia set is not connected and the latter three Julia sets do not satisfy the property that all the boundaries of the Fatou components are pairwise disjoint. None of these Julia sets are Sierpi\'{n}ski carpets.}
  \label{Fig_para-carpet-1}
\end{figure}

\begin{exam}\label{exam-3}
We show that there exists a rational map $f_3$ satisfying the conditions (a) (b) (d) in Theorem \ref{thm:main} but $J(f_3)$ is not a Sierpi\'{n}ski carpet. Let
\begin{equation*}
f_3(z)=\frac{8(z-1)^2(z+8)}{27z}.
\end{equation*}
The critical orbit of $f_3$ is $1\overset{(2)}{\longmapsto} 0\overset{(1)}{\longmapsto}\infty\overset{(2)}{\longmapsto}\infty$ and $-2\overset{(3)}{\longmapsto}-8 \overset{(1)}{\longmapsto}0\overset{(1)}{\longmapsto}\infty$. The Julia set of $f_3$ is connected since it is post-critically finite. Let $U$, $V$ and $W$ be the Fatou components containing $\infty$, $0$ and $1$ respectively. We claim that $U\neq V$. If not, then $U$ is completely invariant and all the critical points are contained in $U$. This means that $J(f_3)$ is disconnected, which is a contradiction. Therefore, we have $U\neq V$, $f_3(U)=U$ and $f_3^{-1}(U)=U\cup V$.

A direct calculation shows that $f_3$ maps $(-\infty,0)$, $(0,1)$ and $(1,+\infty)$ onto $\R$, $(0,+\infty)$ and $(0,+\infty)$ respectively, and the restrictions of $f_3$ on these intervals are strictly monotonous. One can verify that there exist three repelling fixed points $\xi_1$, $\xi_2$ and $\xi_3$ of $f_3$ satisfying $-8<\xi_1<0<\xi_2<1<\xi_3<4$. Actually, $\xi_1=-5.57371629\cdots$ is a \textit{cut point} of the Julia set of $f_3$ and hence $J(f_3)$ is not a Sierpi\'{n}ski carpet. One can refer to \cite[Theorem 6.2]{FY15} for the details of the proof under the similar case (see Figure \ref{Fig_para-carpet-1}).
\end{exam}

\begin{exam}\label{exam-4}
We construct finally a rational map $f_4$ satisfying the conditions (a) (b) (c) in Theorem \ref{thm:main} but $J(f_4)$ is not a Sierpi\'{n}ski carpet. Consider the map
\begin{equation*}
f_4(z)=\lambda_4(z-1)^3/z,
\end{equation*}
where $\lambda_4\in\C\setminus\{0\}$. The map $f_4$ has the critical orbit $1\overset{(3)}{\longmapsto} 0\overset{(1)}{\longmapsto}\infty\overset{(2)}{\longmapsto}\infty$ and a free critical point $-\frac{1}{2}$ with local degree $2$. Let $U$, $V$ and $W$ be the Fatou components containing $\infty$, $0$ and $1$ respectively. Then it is easy to see $f_4^{-1}(U)=U\cup V$, $f_4^{-1}(V)=W$ and $\deg(f_4|_W:W\to V)=3>\deg(f_4|_U:U\to U)=2$.

If one chooses a special $\lambda_4$ such that the free critical point $-\frac{1}{2}$ is periodic under $f_4$ and further, such that $f_4$ is ``renormalizable", then the Julia set of $f_4$ could not be a Sierpi\'{n}ski carpet. Indeed, let $\lambda_4=0.31661929\cdots$ such that $f_4^{\circ 6}(-1/2)=-1/2$ and the Julia set of $f_4$ contains a copy of the ``basilica'' (i.e. the homeomorphic image of the Julia set of $z\mapsto z^2-1$). It is well known that $z\mapsto z^2-1$ has two bounded periodic Fatou components whose boundaries intersect at a point. This means that the Julia set of $J(f_4)$ is connected but not a Sierpi\'{n}ski carpet (see \cite[\S 6]{Ste06} and Figure \ref{Fig_para-carpet-1}).
\end{exam}

\section{Some special carpet Julia sets}\label{sec-special}

In this section, we construct some special carpet Julia sets.  The first one is a carpet Julia set with a parabolic fixed point and the second is a carpet Julia set whose corresponding Fatou set contains a Fatou component with period $2$.

\subsection{A parabolic carpet Julia set}\label{subsec-para}

Let
\begin{equation*}
g_\varrho(z)=\frac{b}{3}\cdot\frac{3z^2+3az-\varrho a}{(z-\varrho)^3},
\end{equation*}
where
\begin{equation*}
a=-\frac{3(2+\varrho)}{9-4\varrho+\varrho^2},\,b=\frac{(1-\varrho)^2 (9-4\varrho+\varrho^2)}{3-2\varrho} \text{ with } \varrho\neq 0,1,\frac{3}{2}
\end{equation*}
are chosen such that $g_\varrho(1)=1$ and $g_\varrho'(1)=1$, i.e. $g_\varrho$ has a parabolic fixed point at $1$ with multiplier $1$. A direct calculation shows that the critical orbits of $g_\varrho$ are
\begin{equation*}
\varrho\overset{(3)}{\longmapsto}\infty\overset{(1)}{\longmapsto}0\overset{(2)}{\longmapsto} \cdots, \text{ and }
z_\varrho:=-2\varrho-2a\overset{(2)}{\longmapsto}\cdots.
\end{equation*}

\begin{prop}\label{prop:parabolic-carpet}
The Julia set of $g_\varrho$ is a Sierpi\'{n}ski carpet if $z_\varrho\not\in W$ but $g_\varrho^{\circ k}(z_\varrho)\in W$ for some $k\geq 1$, where $W$ is the Fatou component containing $\varrho$. In particular, if $\varrho_0=18$, then the Julia set of $g_{\varrho_0}$ is a Sierpi\'{n}ski carpet.
\end{prop}

\begin{proof}
The proof of the first statement is similar to that of Propositions \ref{prop:carpet-3} and \ref{prop:carpet-2}. Suppose that $z_\varrho\not\in W$ but $g_\varrho^{\circ k}(z_\varrho)\in W$ for some $k\geq 1$. Then $g_\varrho$ has a parabolic basin $U$ containing the critical point $0$ with the parabolic fixed point $1$ on its boundary. Let $V$ be the Fatou component containing $\infty$. We claim first that $U\neq V$. Suppose that $U=V$. Then $V=W$ since otherwise each $z\in U$ has at least $2$ preimages in $U$ and $3$ preimages in $W$ (counted with multiplicity) which is a contradiction. Since $g_\varrho^{-1}(\infty)=\varrho$, it means that $U$ is completely invariant under $g_\varrho$. From the assumption that $g_\varrho^{\circ k}(z_\varrho)\in W=U$ for some $k\geq 1$, we have $z_\varrho\in U$, which contradicts the assumption that $z_\varrho\not\in W$. This proves that claim that $U\neq V$. Therefore, we have $V\neq W$ and the critical points $\varrho$ and $z_\varrho$ are not contained in $U$. Since the immediate parabolic basin $U$ contains exactly one critical point $0$, it follows that $U$ is simply connected. By Riemann-Hurwitz's formula, all the preimages of $U$ are simply connected also. Hence $J(g_\varrho)$ is connected and it is a Sierpi\'{n}ski carpet by Theorem \ref{thm:main}.

If $\varrho_0=18$, we have
\begin{equation*}
g_{\varrho_0}(z)=-\frac{289\,(87z^2-20z+120)}{11\,(z-18)^3}.
\end{equation*}
A direct calculation shows that $g_{\varrho_0}^{\circ 3}(z_{\varrho_0})\in(0,1)$, where $z_{\varrho_0}=-3092/87$ is a critical point of $g_{\varrho_0}$. One can verify easily that $g_{\varrho_0}([0,1])\subset[1/2,1]$ and $g_{\varrho_0}$ is increasing strictly on $[0,1]$. This means that all the critical points are attracted by the parabolic basin of $1$. Therefore, in order to prove that $J(g_{\varrho_0})$ is a Sierpi\'{n}ski carpet, it is sufficient to prove that it is connected, or equivalently, to prove that the immediate parabolic basin of $1$ is simply connected. For this, we define an oval disk $B:=\{z=x+y\ii\in\C:(\frac{x+1.5}{4.5})^2+(\frac{y}{5.4})^2<1\}$ and let $A$ be the component of $\C\setminus g_{\varrho_0}(\partial B)$ containing the origin. One can check that $\overline{B}$ is compactly contained in $A$ (This requires some numerical calculations for which we omit the details here. Compare \cite[\S 6]{Ste06} and see the picture in Figure \ref{Fig-carpet-parabolic} on the left). Shrinking the domain $A$ a little if necessary, we assume that $A$ is bounded by a simple closed curve. Let $B'$ be the component of $g_{\varrho_0}^{-1}(A)$ containing $0$. Then $(g_{\varrho_0},B',A)$ is a polynomial-like mapping with degree $2$ (see the definition in \cite{DH85}). In particular, $g_{\varrho_0}$ contains a homeomorphic copy of the Julia set of $z\mapsto z^2+\tfrac{1}{4}$. This means that the immediate parabolic basin of $1$ is simply connected and the Julia set of $g_{\varrho_0}$ is a Sierpi\'{n}ski carpet.
\end{proof}

\begin{figure}[!htpb]
  \setlength{\unitlength}{1mm}
  \includegraphics[height=55mm]{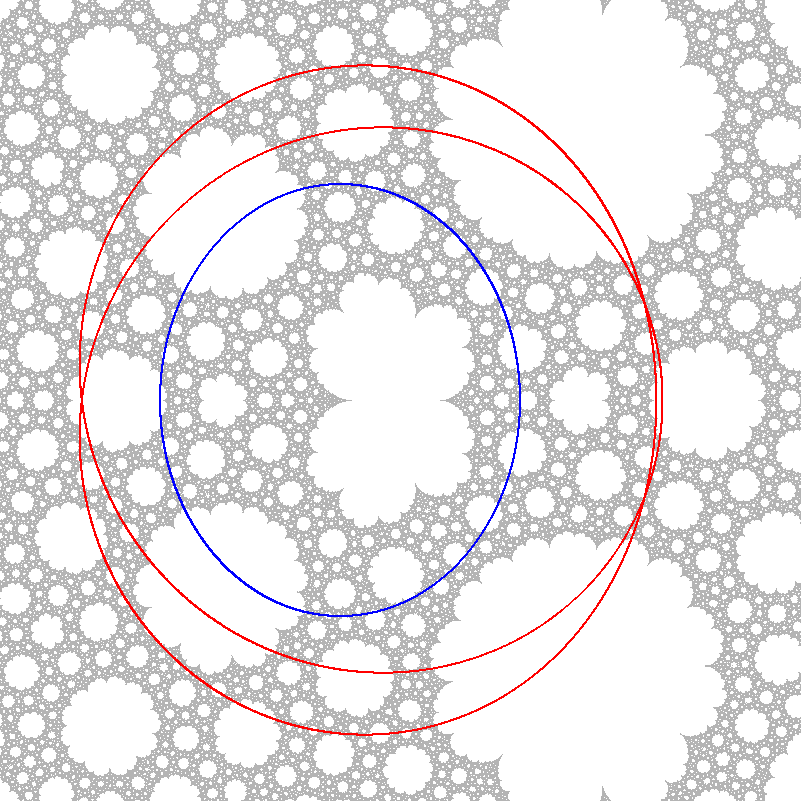}\quad
  \includegraphics[height=55mm]{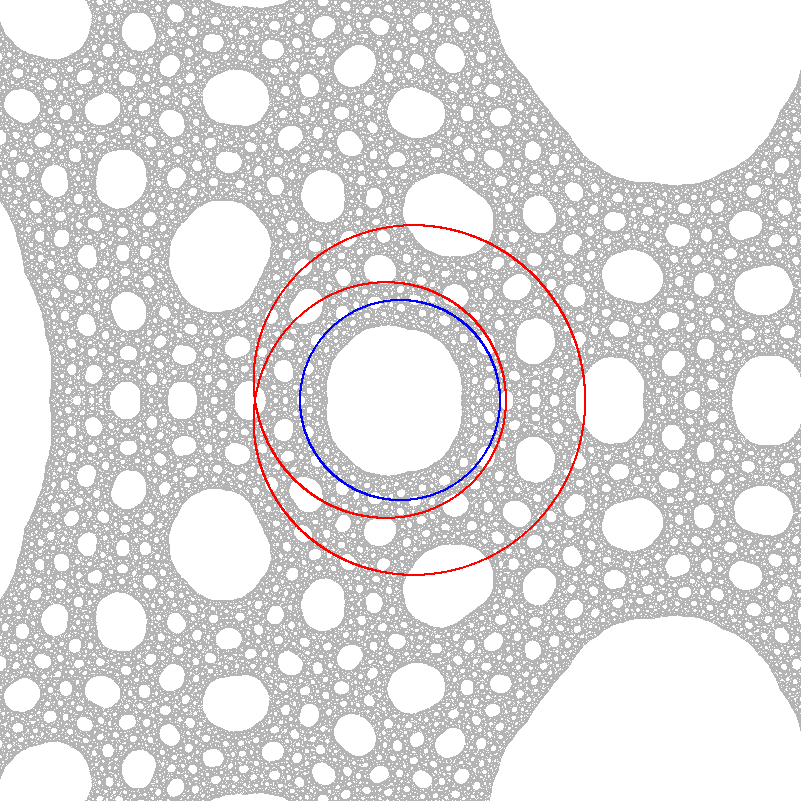}
  \caption{The Julia sets of $g_{\varrho_0}$ and $h_{\alpha_0}$, where $\varrho_0=18$ and $\alpha_0=(1-\sqrt{33})/12$ are chosen such that $g_{\varrho_0}$ has a parabolic fixed point and $h_{\alpha_0}$ has a Fatou component with period $2$. Both Julia sets are Sierpi\'{n}ski carpets. In order to show that some periodic Fatou components are simply connected, both pictures have been depicted the candidates of the domains and the ranges of the polynomial-like mappings.}
  \label{Fig-carpet-parabolic}
\end{figure}

\subsection{A carpet Julia set with Fatou components of period $2$}

In this subsection we construct a carpet Julia set of a cubic rational map with Fatou components of period $2$. Let

\begin{equation*}
h_\alpha(z)=\frac{3\alpha-1}{\alpha^2}\cdot\frac{z^2-3\alpha z+\alpha}{z^2(z-\beta)},
\text{ where }\beta=3-\frac{1}{\alpha} \text{ and }\alpha\neq 0,\frac{1}{3},\frac{1}{2}.
\end{equation*}
A straightforward calculation shows that the critical orbits of $h_\alpha$ are
\begin{equation*}
1\overset{(3)}{\longmapsto}\beta\overset{(1)}{\longmapsto}\infty\overset{(1)}{\longmapsto} 0\overset{(2)}{\longmapsto}\infty
\text{ and } z_\alpha:=6\alpha-2\overset{(2)}{\longmapsto}\cdots.
\end{equation*}
Therefore, $h_\alpha$ has a periodic super-attracting orbit with period $2$.

\begin{prop}
Let $\alpha_0=(1-\sqrt{33})/12=-0.39538022\cdots$ such that $h_{\alpha_0}(z_{\alpha_0})=1$. Then the Julia set of $h_{\alpha_0}$ is a Sierpi\'{n}ski carpet.
\end{prop}

\begin{proof}
Since $h_{\alpha_0}$ is post-critically finite and hyperbolic, the Julia set of $h_{\alpha_0}$ is connected and locally connected.
In order to prove that $J(h_{\alpha_0})$ is a Sierpi\'{n}ski carpet, it is sufficient to prove that each Fatou component is bounded by a simple closed curve and all the boundaries of the Fatou components are pairwise disjoint. The proof will be based on the similar argument as in Proposition \ref{prop:parabolic-carpet}. We define a round disk $B:=\{z=x+y\ii\in\C:|z|<1/10\}$ and let $A$ be the component of $\C\setminus h_{\alpha_0}^{\circ 2}(\partial B)$ containing the origin. One can check that $\overline{B}$ is compactly contained in $A$ and $\overline{B}\cap h_{\alpha_0}(\overline{B})=\emptyset$ (Similarly, we omit the details of the numerical calculations here. See the picture in Figure \ref{Fig-carpet-parabolic} on the right). Shrinking the domain $A$ a little if necessary, we assume that $A$ is a Jordan domain. Let $B'$ be the component of $h_{\alpha_0}^{-2}(A)$ containing the origin. Then $(h_{\alpha_0},B',A)$ is a polynomial-like mapping with degree $2$.

Let $U$ and $V$ be the Fatou components containing $\infty$ and $0$ respectively. According to the Straightening Theorem of Douady and Hubbard \cite{DH85}, we know that $V$ is a Jordan domain. Since all the Fatou components will be iterated to the cycle of the periodic Fatou components $U\leftrightarrow V$ and the Julia set $h_{\alpha_0}$ is connected, it follows that all the Fatou components are Jordan domains \cite[Proposition 2.8]{Pil96}. Since $\overline{B}\cap h_{\alpha_0}(\overline{B})=\emptyset$, it means that $\overline{B}\cap\overline{U}=\emptyset$. By a similar argument to the proof of Theorem \ref{thm:main}, all the boundaries of the Fatou components are pairwise disjoint. Hence the Julia set of $h_{\alpha_0}$ is a Sierpi\'{n}ski carpet.
\end{proof}

\section*{Acknowledgements}

This work is supported by the National Natural Science Foundation of China (grant Nos.\,11401298 and 11671092) and the Fundamental Research Funds for the Central Universities (grant No. 0203-14380013). The author would like to thank the referee for the careful reading and useful suggestions.


\end{document}